\documentclass[12pt]{amsart}
\usepackage{mathrsfs}
\usepackage{amsthm}
\usepackage{amssymb}
\usepackage{amsmath}
\numberwithin{equation}{section}
\newcommand{\bea}{\begin{eqnarray}}
\newcommand{\eea}{\end{eqnarray}}
\newcommand{\be}{\begin{eqnarray*}}
\newcommand{\ee}{\end{eqnarray*}}
\newtheorem{theorem}{Theorem}[section]
\newtheorem{lemma}{Lemma}[section]

\newtheorem{definition}{Definition}[section]

\newtheorem{example}{Example}[section]

\begin{document}
\title[Nil Evolution Algebras]{Finitely Generated Nil but Not Nilpotent Evolution Algebras}
\author[J. P. Tian]{Jianjun Paul Tian}
\address{Mathematics Department, College of William and Mary, Williamsburg, VA 23187, USA} \email{jtian@wm.edu}
\author[Y. M. Zou]{Yi Ming Zou}
\address{Department of Mathematical Sciences, University of Wisconsin-Milwaukee, Milwaukee, WI 53201, USA} \email{ymzou@uwm.edu}
\thanks{AMS 2000 MSC: 17D92, 92C15}
\begin{abstract}
To use evolution algebras to model population dynamics that both allow extinction and introduction of certain gametes in finite generations, nilpotency must be built into the algebraic structures of these algebras with the entire algebras not to be nilpotent if the populations are assumed to evolve for a long period of time. To adequately address this need, evolution algebras over rings with nilpotent elements must be considered instead of evolution algebras over fields. This paper develops some criteria, which are computational in nature, about the nilpotency of these algebras, and shows how to construct finitely generated evolution algebras which are nil but not nilpotent.   
\end{abstract}
\maketitle
\section{Introduction}
\par
The non-associative algebras that naturally arise as one considers the processes of genetic information getting passed down and evolving through the generations are usually called genetic algebras \cite{Et1, Et2, R, Sch, T, W}. These algebras can be defined by a basis $\{x_{i}:i\in I\}$ over a commutative ring\footnote{The existing theory on these algebras was developed mainly over a field. We will explain the need to consider these algebras over general commutative rings in Section 2.} $R$, where $I$ is a fixed index set which usually is taken to be finite, with the multiplication defined by
\bea\label{e1}
x_{i}x_{j}=\sum_{k\in I}c_{kij}x_{k}, \quad c_{kij}\in R, \quad i,j\in I.
\eea
When these algebras are used to describe a biological population, the basis elements $x_{i},\;i\in I$, represent the gametes of the corresponding population, and the product of two gametes represents the reproduction process. Depending on the extra conditions that one imposes on the coefficients $c_{kij}$, these algebras can be used to model different biological populations and are named differently in the literature. For example, when the coefficients satisfy the extra condition:
\bea
c_{kij}=0,\quad \forall\; i\ne j,
\eea
the resulting algebras are the so-called {\it evolution algebras} \cite{T}, which are used to describe biological systems of non-Mendelian inheritance.
\par
Let $A$ be a genetic algebra with a set of generators $X$. Then the elements in $X$ can be viewed as the original gametes of the population. If some of these original gametes extinct after, say $N$ generations, then the subalgebra generated by $X^{n}$ (see Section 2 for definition) will not contain these gametes for all $n>N$. On the other hand, if new gametes are introduced through the process of evolution, say the first new gamete appears at generation $M$, then the subalgebra generated by $X^{M}$ is not a subspace of the space spanned by original generators in $X$. Therefore, certain nilpotent property is needed if some of the original gametes (generators) disappear later on, and a sufficiently large space is needed to allow the introduction of new gametes for the entire evolution process. It is not uncommon that every gamete extincts after certain generations but the whole population will evolve for a  long period of time, so that we may assume that the population will not extinct. In mathematics, these properties imply that the algebra is nil but not nilpotent. Though there have been studies devoted to these algebras when the coefficient ring is a field, the study of these algebras over an arbitrary ring, to our best knowledge, seems to be virtually nonexistent.  
\par
For evolution algebras over a field, say the real numbers, the extinction of gametes can be handled in two ways \cite{T}: set the coefficients of the corresponding gametes to $0$, so that these gametes are not reproduced in the next generation; or take the limit and let the process go to infinity so that the coefficients of these gametes approach $0$. However, these approaches do not adequately address the important intermediate cases where gametes extinct in finite generations. We will show that nilpotent elements from the base ring is necessary if the extinction process happens in the later generations instead of the second generation. If the associativity holds for the algebra\footnote{As usual, by ``non-associative algebras'' we mean that the associativity is not assumed. For algebras such that associativity fails, we use ``not associative''.}, then the construction of a finitely generated nil but not nilpotent infinite dimensional algebra with finite Gelfand-Kirillov dimension seems to be difficult if the coefficients are taken over a field \cite{LS}. In this short note, our main goal is to address the nil but not nilpotent property motivated by biological systems mentioned before, and we will limit our discussions to the nilpotency properties of evolution algebras over an arbitrary commutative ring. We will first discuss some basic properties on the nilpotency of these evolution algebras, in such a way that these conditions can be checked computationally. Then we will explain how to construct finitely generated nil but not nilpotent evolution algebras over commutative rings that have nilpotent elements. 
\par 
\section{Nilpotent Property}
\par
Let $A$ be a generic algebra with a set of generators $X$ as defined in section 1. We call the $c_{kij}$'s that appear on the right hand side of (\ref{e1}) the {\it structure coefficients} of $A$. We also call the cardinality of $X$ the dimension of $A$, though we should use the term ``rank'' for a free module over a general commutative ring.
\par
The definition of $A$ implies that in general $A$ is not power-associative and non-commutative. For further properties and examples of these algebras when the coefficient ring $R$ is a field, we refer the reader to references \cite{Sch1, T, W}.
\par
We need the notion of {\it principal powers} and {\it plenary powers} of an element in a non-associative commutative algebra. Let $a\in A$, then the principal powers of $a$ are defined by
\be
a, a^{2},  \ldots, a^{i}=a^{i-1}a,\ldots;
\ee
and the plenary powers of $a$ are defined by
\be
a^{[1]}=a^{2}, a^{[2]}=(a^{[1]})(a^{[1]}), \ldots, a^{[i]}=a^{[i-1]}a^{[i-1]}, \ldots.
\ee
We also define the principal powers of the algebra $A$ to be
\be
A, A^{2},\ldots, A^{i}=A^{i-1}A, \ldots.
\ee
\par
\begin{definition}
We call $A$ a nil algebra if for every $a\in A$, $a^{n}=0$ for some $n\in \mathbb{N}$, and we call $A$ nilpotent if $A^{n}=(0)$ for some $n\in \mathbb{N}$.
\end{definition}
\par
Note that in a non-associative algebra, $a^{n}=0$ does not imply that $a^{[n]}=0$ in general. Note also that our definition of a nilpotent algebra is different than the one in \cite{Sch1}, which we term as {\it strongly nilpotent}.
\begin{definition} We call $A$ strongly nilpotent if there exists an integer $n>0$ such that any product $a_1a_2\cdots a_n$ of $n$ elements in $A$, no matter how associated, is $0$.
\end{definition} 
\par
According to the definitions, strongly nilpotent implies nilpotent, but not conversely. 
\par
There is a well-known result on strongly nilpotent algebras over a field. To state this result, for each $a\in A$, we define a linear transformation of $A$ by
\be
L_a : x\longrightarrow ax,\quad \forall\; x\in A.
\ee
We let $\mathcal{L}(A)$ be the associative algebra generated by all $L_a,\;a\in A$, and call it the {\it associated algebra} of $A$. In the case of an evolution algebra, due to the commutativity of the multiplication in $A$, the algebra $\mathcal{L}(A)$ is the same as the associated algebra generated by all the right multiplications defined by the elements of $A$. We have
\bea
(L_aL_b)(x)=L_a(L_b(x)),\quad\forall\; a,b,x\in A.
\eea
However, $L_{ab}\ne L_aL_b$ in general. For each generator $x_i$, we abbreviate $L_{x_i}$ as $L_i$. If $a=\sum_{i}a_ix_i$, then we have
\bea
L_a = \sum_{i}a_iL_i.
\eea
\par
 The following theorem is contained in \cite{Sch1}.
\begin{theorem} 
If the base ring $R$ is a field, the algebra $A$ is strongly nilpotent if and only if the associated algebra $\mathcal{L}(A)$ is nilpotent.
\end{theorem}
\par\medskip
From now on, we assume our algebra $A$ is an evolution algebra, i.e. $c_{kij} = 0$ for all $i\ne j$, and we simplify our writing of the structure coefficients to $c_{ki}$.
\par
To motivate our study for evolution algebras over arbitrary commutative rings, we make a simple observation.  
\begin{lemma} Let $A$ be nil. Then for each $1\le i\le n$, there exists a positive integer $k_i$ such that $c_{ii}^{k_i}=0$.
\end{lemma}
\begin{proof} We have the following formula:
\be
x_i^k=c_{ii}^{k-2}x_i^2,\quad \mbox{for all integers $k>2$},
\ee
which holds in any evolution algebra. If $A$ is nil, then each generator $x_i$ is nilpotent, and the lemma follows from the above formula.
\end{proof}
\par
This lemma implies that if the coefficient ring $R$ is a domain, then in order for $x_i$ to be nilpotent, it is necessary that $c_{ii}=0$. In biology, this means that a gamete does not reproduce itself. To model the situations where $x_i$ reproduces itself for a number of generations but disappears later on, the coefficient ring must contain nilpotent elements. In the rest of this section, we will describe some criteria about the nilpotency of an evolution algebra based on the structure coefficients.
\par
We now derive a criterion for checking whether a finite dimensional evolution algebra is nil. Let 
\be
X=\{x_1,x_2,\ldots,x_N\}.
\ee
 Then we have
\bea
(x_1^2,x_2^2,\ldots,x_N^2) &=& (x_1,x_2,\ldots,x_N)(c_{ki})_{N\times N}\\\nonumber
                       {}  &:\;=& (x_1,x_2,\ldots,x_N)C,
\eea
where the $j$th column of the $N\times N$ matrix $C$ is $(c_{1j},c_{2j},\ldots, c_{Nj})^T$. For each $\alpha=(a_1,a_2,\ldots,a_N)\in R^N$, define an $N\times N$ matrix $C_{\alpha}$ by multiplying $a_j$ to the $j$th column of the structure coefficient matrix $C$:
\bea
C_{\alpha}=(a_jc_{kj}).
\eea 
\begin{theorem} If $X=\{x_1,x_2,\ldots,x_N\}$, then the evolution algebra $A$ is nil if and only if for every $\alpha\in R^N$, there exists a positive integer $k_{\alpha}$ such that
\bea\label{t1}
C_{\alpha}^{k_{\alpha}}\alpha^T=0.
\eea
\end{theorem}
\begin{proof} Each $a\in A$ can be written as
\be
a=(x_1,x_2,\ldots,x_N)\alpha^T
\ee
for some $\alpha\in R^N$. We use induction on $n$ to prove the following formula:
\bea\label{p1}
a^n=(x_1,x_2,\ldots,x_N)C_{\alpha}^{n-1}\alpha^T,\;\;\mbox{$\forall\; n\ge 2$}.
\eea
For $n=2$, we have
\be
a^2 &=& \sum_{i=1}^Na_i^2x_i^2=\sum_{i,k=1}^Na_i^2c_{ki}x_k=\sum_{i,k=1}^Nx_k(c_{ki}a_i)a_i\\
 {} &=& (x_1,x_2,\ldots,x_N)C_{\alpha}\alpha^T.
\ee
Thus the formula holds in this case. Assume that the formula holds for $n\ge 2$. Then
\be
a^{n+1} &=& (a^n)a=((x_1,x_2,\ldots,x_N)C_{\alpha}^{n-1}\alpha^T)a\\
     {} &=& ((x_1,x_2,\ldots,x_N)\beta^T)a,
\ee
where 
\be
\beta^T=(b_1,b_2,\ldots,b_N)^T :=C_{\alpha}^{n-1}\alpha^T.
\ee
Now
\be
a^{n+1} &=& (\sum_{i=1}^Nb_ix_i)(\sum_{j=1}^Na_jx_j)=\sum_{i=1}^Nb_ia_ix_i^2\\
     {} &=& \sum_{i=1}^Nb_ia_i\sum_{j=1}^Nc_{ji}x_j\\
     {} &=& (x_1,x_2,\ldots,x_N)C_a\beta^T\\
     {} &=& (x_1,x_2,\ldots,x_N)C_a^n\alpha^T.
\ee
Thus (\ref{p1}) holds for all $n\ge 2$. Now the theorem follows from (\ref{p1}).
\end{proof}
\par
As for nilpotent evolution algebras, we have
\begin{theorem} The evolution algebra $A$ (finite or infinite dimensional) is nilpotent, i.e. $A^n=(0)$ for some positive integer $n$, if and only if for any sequence of indexes $i_1, i_2, \ldots, i_n$ the structure coefficients satisfy
\bea\label{t2}
c_{i_ni_{n-1}}c_{i_{n-1}i_{n-2}}\cdots c_{i_2i_1}=0.
\eea
\end{theorem}
\begin{proof}
If $A^n=(0)$, then we have 
\bea\label{p2}
\underbrace{(((ab)c)\cdots)}_{\mbox{$n$ terms}}=0,\quad a,b,c,\ldots \in A.
\eea
In particular, we have for any sequence of indexes $i_1, i_2, \ldots, i_{n-1}$
\bea\label{p3}\quad
(((x_{i_1}^2)x_{i_2})\cdots x_{i_{n-1}})=c_{i_{n-1}i_{n-2}}\cdots c_{i_2i_1}\sum_{k}c_{ki_{n-1}}x_k=0,
\eea
which implies (\ref{t2}). Conversely, (\ref{t2}) implies (\ref{p3}). Since the left hand side of (\ref{p2}), after expressing each term as a linear combination of the $x_i$'s and multiplying out, is a sum of terms similar to the one on the left hand side of (\ref{p3}), it must equal to $0$, i.e. $A$ is nilpotent.
\end{proof}
\par
Let
\be
I_1=span_R\{x_i\in X\;|\;x_i^2=0\}.
\ee
Then $I_1$ is an ideal of $A$ and $AI_1=(0)$.
\begin{lemma}
The evolution algebra $A$ is nilpotent if and only if the quotient algebra $A/I_1$ is nilpotent.
\end{lemma}
\begin{proof}
If $A$ is nilpotent, then all its quotients are nilpotent. If $A/I_1$ is nilpotent, then there is an $n$ such that $(A/I_1)^n=(0)$, which implies that $A^n\subseteq I_1$. By the comment just before the lemma, $A^{n+1}\subseteq AI_1=(0)$, so $A$ is nilpotent.
\end{proof}
\par
Note that $A/I_1$ is spanned by the images of those $x_i$ such that $x_i^2\ne 0$. If 
\be
X'=\{x_i\;|\; x_i^2\ne 0\;\mbox{and}\; x_i^2\in I_1\}\ne \emptyset
\ee
then 
\be
I_2=span(I_1\cup X')\supsetneq I_1
\ee
is an ideal of $A$ and $I_2^2\subseteq I_1$. Continuing this way, we get a filtration of ideals of $A$:
\be
A\supset\cdots\supset I_2\supset I_1\supset I_0=(0),
\ee 
such that $(I_{i+1}/I_i)^2=(0)$ for $i\ge 0$. If $A$ is finite dimensional, then there are two possibilities: The above process produces a complete filtration, i.e. 
\be
A=I_k\supset I_{k-1}\supset \cdots\supset I_2\supset I_1\supset I_0=(0)
\ee 
such that $(I_{i+1}/I_i)^2=(0)$ for $i\ge 0$; or for some index $s$, $A/I_s\ne (0)$ and 
\be
\{x_i\;|\; x_i^2\ne 0\;\mbox{and}\; x_i^2\in I_s\} = \emptyset.
\ee
\par
In the first case, we rearrange the generators $x_i,\;1\le i \le N$, if necessary, such that 
\be
I_0 &=& span\{x_1,\ldots, x_{i_0}\},\\
I_1 &=& span\{x_{i_0+1},\ldots, x_{i_1}\},\\
{} &\vdots & \\
I_k &=& span\{x_{i_{k-1}+1},\ldots, x_{i_k}\}.
\ee
Under this order of the generators, the structure coefficient matrix $C$ of $A$ is a strict upper triangular matrix.
\par
In the second case, let $B=A/I_s$ and let $y_i=\overline{x}_i$ be the image of $x_i$ in $B$. Then $B$ is spanned by
\be
Y=\{y_i\;|\; x_i\notin I_s\}:=\{y_1,y_2,\ldots,y_m\},
\ee
and $y_i^2\ne 0,\;\forall\; y_i\in Y$. If in addition we have that the coefficient ring $R$ is a domain, then $Y$ is not nilpotent. This can be seen as follows. Since $y_1^2\ne 0$, there is a $c_{i_11}\ne 0$, thus
\be
(y_1^2)y_{i_1}=c_{i_11}y_{i_1}^2.
\ee
Since $y_{i_1}^2\ne 0$, there is a $c_{i_2i_1}\ne 0$, and so on, we can obtain a sequence of nonzero elements of arbitrary length such that their product is nonzero since $R$ is a domain. Thus by Theorem 2.4, $B$ is not nilpotent. Summarizing our discussion, we have the following theorem.
\begin{theorem}
If $R$ is a domain and $A$ is finite dimensional over $R$, then $A$ is nilpotent if and only if there is an ordering of the generators $x_i,\;1\le i\le N$, such that under this ordering, the structure coefficient matrix $C$ of $A$ is strictly upper triangular.
\end{theorem}
\par\medskip
\section{Examples}
Using the theorems of Section 2, we can construct examples of evolution algebras with desired properties easily.
\par
\begin{example}
Let $R=\mathbb{Z}_{36}$, let $X=\{x_1, x_2\}$, and let the structure coefficient matrix be
\be
\left(\begin{array}{cc}
6 & 3\\
2 & 12
\end{array}
\right).
\ee
Then the evolution algebra $A$ is nilpotent: $A^5=(0)$.
\end{example}
\par
\begin{example}
Again, let $R=\mathbb{Z}_{36}$, let $X=\{x_1, x_2\}$. But let the structure coefficient matrix be
\be
\left(\begin{array}{cc}
6 & 2\\
2 & 12
\end{array}
\right).
\ee
Then the evolution algebra $A$ is not nilpotent, since the product 
\be
\underbrace{c_{21}c_{12}c_{21}c_{12}\cdots}_{\mbox{$n$ terms}}=2^n
\ee
is never zero, so Theorem 2.3 implies the result. In fact, $A$ is not even nil since $a = x_1+x_2$ is not nilpotent.
\end{example}
\par
\medskip
For associative algebras over a field, the problem of constructing a finitely generated infinite dimensional nil but not nilpotent algebra with finite Gelfand-Kirillov dimension seems to be complicated: the sole purpose of \cite{LS} is to construct such an example. However, it is quite easy to construct a nil but not nilpotent infinite dimensional evolution algebra which is singly generated over a ring. In view of possible interests from different applications, we give such an example next. This algebra $A$ is defined over $R=\mathbb{Z}_{4}$ with basis $X=\{x_{i}:i\in \mathbb{N}\}$ and with the defining relations (\ref{e1}) specified to:
\bea\label{e3}
x_{i}x_{j}=\left\{ \begin{array}{rl} 
2x_{i}+x_{i+1},& i=j,\\
0, & i\ne j; \end{array}\right. \quad \forall\; i,j\in \mathbb{N}.
\eea
\par
Note that from (\ref{e3}), we immediately see that the subalgebra of $A$ generated by $x_1$ is equal to $A$.
\par
The choice of the base ring $\mathbb{Z}_{4}$ is for the reason of simplicity of the presentation here. It will be clear that the base ring can be replaced by other commutative rings with nilpotent elements. In particular, one can take $\mathbb{Z}_2^{m}$, where $m$ is any integer $\ge 2$, as the base ring. We will provide another example at the end of section 4.
\begin{theorem}
The evolution algebra $A$ defined over $\mathbb{Z}_4$ with basis $X=\{x_i\;:\;i\in\mathbb{N}\}$ subjected to relation (\ref{e3}) is a nil but not nilpotent algebra generated by $x_{1}$. Furthermore, $A$ is not associative.
\end{theorem}
\begin{proof}
\par
First we show that $A$ is not associative. By definition we have 
\be
x_{1}^{2}=2x_{1}+x_{2},\quad x_{1}^{3}=2x_{2}, \quad x_{1}^{4}=0,
\ee
and
\be
(x_{1}^{2})(x_{1}^{2})=(2x_{1}+x_{2})(2x_{1}+x_{2})=x_{2}^{2}=2x_{2}+x_{3}.
\ee 
Thus the algebra $A$ is not associative (in fact, not even power-associative).
\par
Next we show that $A$ is generated by $x_{1}$. More precisely, we claim that
\be
x_{1}^{[n]}=2x_{n}+x_{n+1},\quad\forall\; n\ge 1.
\ee
This is true for $n=1$ by the definition of the plenary powers. Assume that it holds for $n$, then 
\be
x_{1}^{[n+1]} &=&(x_{1}^{[n]})(x_{1}^{[n]})\\
&=&(2x_{n}+x_{n+1})(2x_{n}+x_{n+1})\\
&=&x_{n+1}x_{n+1}=2x_{n+1}+x_{n+2}.
\ee
Thus $A$ is generated by $x_{1}$. Note this also shows that $A$ is not nilpotent.
\par
Now we show that $A$ is nil. For an arbitrary element $a\in A$, we can always write 
\bea\label{e4}
a=\sum_{1\le i\le t}a_{i}x_{i},
\eea
where $a_{i}\in R$ (since we can always insert terms with zero coefficients). For $j\ge 1$, let
\be
A_{j}=span\{x_{i}:i\ge j\}.
\ee
Then each $A_{j}$ is an ideal of $A$ and 
\bea\label{e5}
& A_{j}x_{i}=(0), \quad \forall\; 1\le i<j.
\eea
We claim that for the element $a$ specified by (\ref{e4}) we have
\bea\label{e6}
a^{3+2j}\in A_{j+2},\quad \forall\; j\ge 0.
\eea
The following computations show that (\ref{e6}) holds for $j=0$:
\be
a^{2} \equiv  (a_{1}x_{1})^{2}\; \mbox{(mod $A_{2}$)}\;
= a_{1}^{2}(2x_{1}+x_{2}),
\ee
\be
a^{3} \equiv 2a_{1}^{3}x_{2}\equiv 0 \quad \mbox{(mod $A_{2}$)}.
\ee
Assume (\ref{e6}) holds for $j=k$. Then we have
\be
a^{3+2k}\equiv b_{k+2}x_{k+2} \quad \mbox{(mod $A_{k+3}$)},
\ee
for some $b_{k+3}\in R$. Hence
\be
& & a^{3+2(k+1)} =((a^{3+2k})a)a\\
&\equiv& ((b_{k+2}x_{k+2})(a_{k+2}x_{k+2})) (a_{k+2}x_{k+2}) \quad \mbox{(mod $A_{k+3}$)}\\
&\equiv&   0 \quad \mbox{(mod $A_{k+3}$)}.
\ee
Therefore (\ref{e6}) holds as claimed. Now let $n=2t+1$, then (\ref{e6}) implies that $a^{n}\in A_{t+1}$. Thus (\ref{e5}) implies $a^{n+1}=0$. This completes the proof of Theorem 3.1.
\end{proof}
\par\medskip
The above arguments work for arbitrary commutative rings with nilpotent elements, for example, we have:
\begin{example}
Let $\mathbb{R}$ be the set of real numbers and let $\mathbb{R}[t]$ be the set of polynomials with real coefficients. Consider $R=\mathbb{R}[t]/(t^2)$. Then the evolution algebra $A$ defined over $R$ with basis $X=\{x_{i}:i\in \mathbb{N}\}$ and the defining relations:
\bea
x_{i}x_{j}=\left\{ \begin{array}{rl} 
tx_{i}+x_{i+1},& i=j,\\
0, & i\ne j; \end{array}\right. \quad \forall\; i,j\in \mathbb{N},
\eea
is nil but not nilpotent.
\end{example}
\par\medskip
It is now clear that one can construct finitely generated nil but not nilpotent evolution algebras with any number of generators: since for an evolution algebra, $x_ix_j = 0$ if $i\ne j$, one can easily combine singly generated nil but not nilpotent evolution algebras together to form the algebra desired. 
\medskip
\section{Concluding Remark}
Genetic algebras over fields have been studied in detail in the literatures. In modeling biological systems, one also needs to consider genetic algebras over arbitrary commutative rings. Our goal here is to address the need of using evolution algebras to model populations such that certain gametes extinct in the process but the population evolves for a long period of time. This cannot be adequately addressed by using evolution algebras over fields. This motivates our consideration of evolution algebras over general commutative ring, and leads to our construction of finitely generated nil but not nilpotent evolution algebras over commutative rings which have nilpotent elements. A general theory about genetic algebras over arbitrary commutative rings is desirable for applications and deserves further attention.

\end{document}